\documentclass[11pt, a4paper]{amsart}
\usepackage{amsmath}
\usepackage{amssymb}
\usepackage{amsthm}
\usepackage[figuresright]{rotating}
\usepackage[all]{xy}
\usepackage{pdfsync}
\usepackage[toc,page]{appendix}
\usepackage{hyperref}
\usepackage{mathtools}
\usepackage{lipsum}
\usepackage[T1]{fontenc}        
\usepackage[utf8]{inputenc}
\usepackage[toc,page]{appendix}
\usepackage{tikz-cd}
\usepackage{accents}
\usepackage{fullpage}
\usepackage{graphicx}
\usepackage{tabularx}
\usepackage{booktabs}
\usepackage{float}  
\usepackage[nameinlink]{cleveref} 
\usepackage{caption}

\newcommand{\jac}{\mathrm{Jac(\mathcal C)}}

\newcommand{\eq}[1]{\begin{equation}#1 \end{equation}}
\newcommand{\eqsp}[1]{\begin{equation} \begin{split}
 #1\end{split}\end{equation}}
\newcommand{\eqst}[1]{\begin{equation*}#1\end{equation*}}

\theoremstyle{plain}
\newtheorem{theorem}{Theorem}[section]
\newtheorem*{theorem*}{Theorem}
\newtheorem{proposition}{Proposition}[section]
\newtheorem{lemma}{Lemma}[section]
\newtheorem{corollary}{Corollary}[section]
\newtheorem{example}{Example}[section]
\newtheorem{question}{Question}[section]
\newtheorem{remark}{Remark}[section]
\newtheorem{definition}{Definition}[section]

\numberwithin{equation}{section}

\begin{document}
\author{Bora Yalkinoglu} 
\address{CNRS and IRMA, 7 rue René-Descartes, 67100 Strasbourg}
\email{yalkinoglu@math.unistra.fr}
\date{\today}
\title{Arithmetic aspects of discrete periodic Toda flows}
\maketitle

\begin{abstract}
\noindent
We construct a new algebraic linearization of the discrete periodic Toda flow by using Mumford's algebraic description of the Jacobian of a hyperelliptic curve. In particular, the discrete periodic Toda flow can be expressed in terms of the famous Gauß composition law for quadratic forms adapted to the framework of hyperelliptic curves by Cantor. One remarkable consequence of our approach is a new integrality property for the discrete periodic Toda flow. This leads, in particular, to a $p$-adic description of the closely related periodic box-ball flow, which has very surprising connections to number theory.
\end{abstract}

\section{Introduction}

\subsection{Motivation}
\label{sec-motivation}
The periodic box--ball system is an integrable cellular automaton introduced by Tokihiro and Yura \cite{yura}.
It evolves configurations of $N>2$ boxes arranged on a circle, where each box is either empty (0) or contains a ball (1),
under the constraint that the total number of balls is $<N/2$. Thus the phase space is
\[
\mathcal V_N=\Bigl\{y\in\{0,1\}^N \ \Big|\ \sum_{i=1}^N y_i < \tfrac{N}{2}\Bigr\},
\]
and the time evolution $\mathcal B:\mathcal V_N\to \mathcal V_N$ moves each ball to the nearest empty box to the right, with occupied boxes forbidden.
Since $\mathcal V_N$ is finite and $\mathcal B$ is invertible, every state $y\in \mathcal V_N$ is periodic; we write $\mathrm{per}(y)\ge 1$ for the minimal period.

\begin{figure}[ht]
  \centering
  \includegraphics[width=0.7\linewidth]{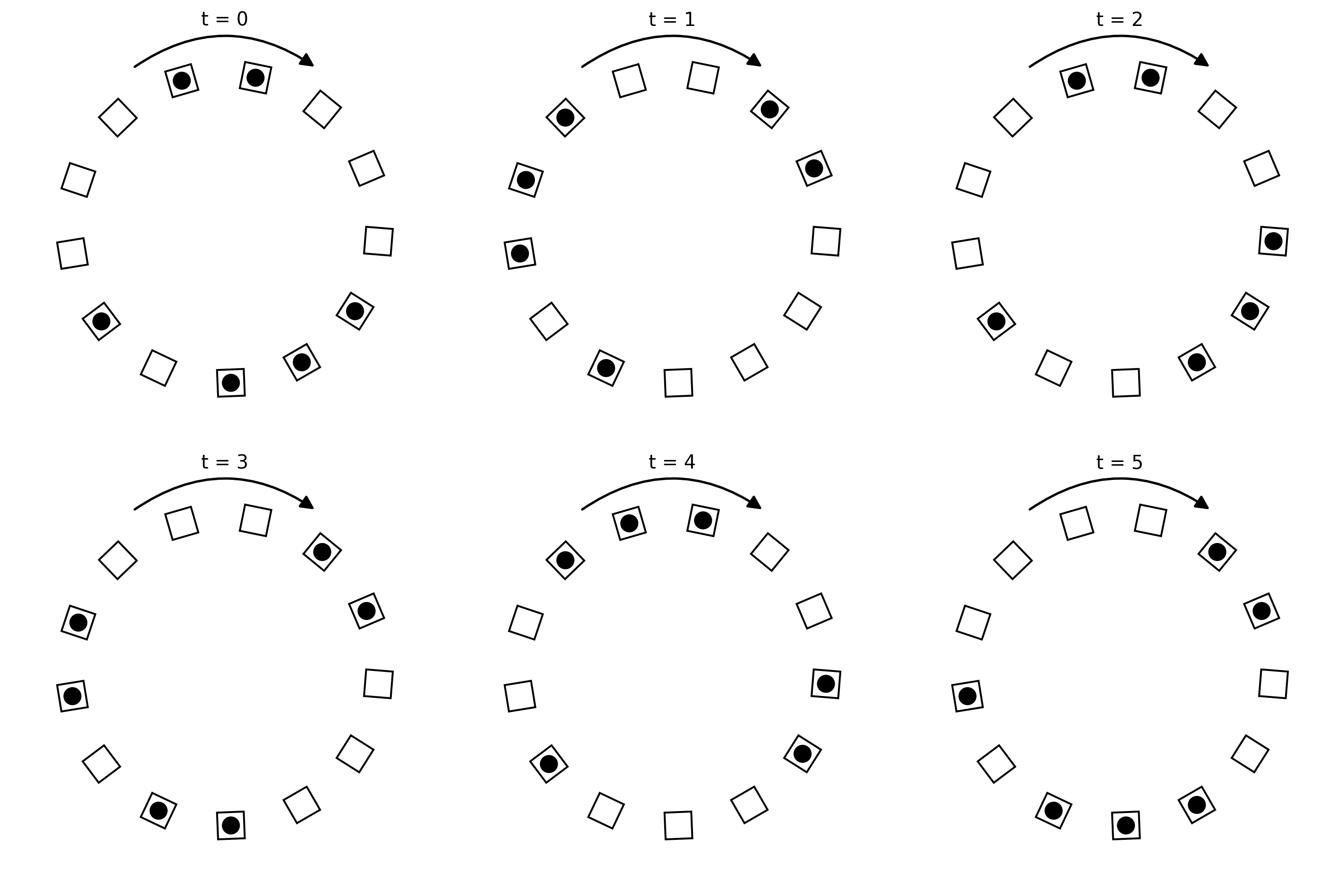}
  \caption{Evolution of the state $y_{13} = 1100011000101 \in \mathcal V _{13}$ with $\mathrm{per}(y_{13}) = 273$.}
  \label{pbb_pic}
\end{figure}
\noindent A remarkable discovery of Tokihiro and Mada shows that the distribution of these periods is linked to deep number theory:
roughly speaking, for each $N$ one can find a state whose period grows like $\exp(\psi(N))$, where $\psi(k) = \mathrm{log}(\mathrm{lcm}(1,2,\dots,k))$ is the Chebyshev's famous second function.
As explained in \cite{riemann}, this viewpoint allows one to formulate the Riemann hypothesis in terms of asymptotic properties
of the periodic box--ball dynamics. \\ \\ 
\noindent 
These connections motivate a systematic study of the geometric and arithmetic structures underlying periodic box--ball systems.
A key bridge is provided by the tropical periodic Toda system, which is closely related to the box--ball system and admits a lift
to the discrete periodic Toda system. In this paper we therefore study the discrete periodic Toda flow and show how the periodic box--ball
flow can be recovered from it.

\subsection{Overview of results.}
The main idea of this paper is to construct a new \emph{algebraic} linearization of the discrete periodic Toda flow using
Mumford's description of the Jacobian $\jac$ of the associated hyperelliptic spectral curve $\mathcal C$.
The group law on $\jac$ is implemented by Cantor's adaptation of Gau{\ss}' composition law for quadratic forms to the
hyperelliptic setting. \\
\noindent
The core technical input is the eigenvector (or algebraic Abel-Jacobi) map
\[
\Psi_C : \mathcal R_\mathcal C \longrightarrow \jac,
\]
which assigns to a Toda state on the isolevel set $\mathcal R_\mathcal C$ a Mumford representation in $\mathrm{Jac}(\mathcal C)$.
While it is classical that discrete periodic Toda flows linearize on Jacobians (typically via $\theta$-functions on the analytic Jacobian),
our approach stays entirely algebraic and yields several new structural and arithmetic consequences:

\begin{itemize}
\item[]
  \item \textbf{Explicit inverse Abel--Jacobi map.}
  We obtain a transparent algebraic description of $\Psi_\mathcal C^{-1}$ (\textbf{\hyperref[inversepsi]{Proposition~\ref*{inversepsi}}}), giving explicit recovery formulas for Toda variables
  from a Mumford representation.
\item[] 
  \item \textbf{Cyclic shift $\sigma$ and torsion.}
  The cyclic shift map $\sigma$ on the Toda phase space corresponds under $\Psi_\mathcal C$ to translation by a distinguished torsion point
  $\mathfrak{pa}\in \mathrm{Jac}(\mathcal C)$ arising from a fundamental solution of the Pell--Abel equation of $\mathcal C$ (\textbf{\hyperref[torsionpoint]{Theorem~\ref*{torsionpoint}}}).
\item []
  \item \textbf{Linearization of the discrete Toda flow.}
  On $\mathrm{Jac}(C)$ the Toda time evolution is translation by an element $\mathfrak T\in \mathrm{Jac}(\mathcal C)$ (\textbf{\hyperref[lineartime]{Theorem~\ref*{lineartime}}}),
  giving a concrete description of the discrete periodic Toda flow in terms of the Gau{\ss}--Cantor composition law.
\item[]
\item \textbf{Discrete Toda flow via $\sigma$.} The discrete periodic Toda flow can be expressed in terms of the cyclic shift $\sigma$ (\textbf{\hyperref[cyclictime]{Proposition \ref*{cyclictime}}}).
\item[]
  \item \textbf{Why $\mathbb{Q}(q)$ is natural.}
  Although Toda flows are traditionally studied over $\mathbb{R}$ or $\mathbb{C}$, our framework highlights that working over the rational
  function field $\mathbb{Q}(q)$ is much more natural and fundamental. In particular, via the $q$-adic valuation this setting recovers both the tropical
  periodic Toda flow and the periodic box--ball flow (\textbf{\hyperref[reformulation]{Proposition~\ref*{reformulation}}} \& \textbf{\hyperref[maintheorem2]{Theorem~\ref*{maintheorem2}}}).
\item[]
  \item \textbf{Integrality and $p$-adic Toda flow.}
  A common belief in the field of integrable systems is that the discrete periodic Toda flow is only defined over a field. \\
  Surprisingly, we prove that the discrete periodic Toda flow is defined over local rings such as $\mathbb{Z}[q]_{(q)}$ (\textbf{\hyperref[integrality]{Theorem~\ref*{integrality}}}).
  Specializing $q\mapsto p$ then yields $p$-adic Toda flows and recovers the periodic box--ball flow via the $p$-adic valuation (\textbf{\hyperref[padic]{Corollary~\ref*{padic}}}).
  \item[]
\end{itemize}

\noindent Finally, this arithmetic viewpoint suggests that $p$-adic methods (for instance those used in \cite{eid}) may help extend Cantor's division-polynomial theory from imaginary hyperelliptic curves to the real hyperelliptic curves that naturally occur as Toda spectral curves (cf., \cite{yalkinoglu}).

\subsection{Organization of the paper.}
Section~2 introduces the discrete periodic Toda flow, the associated spectral curve, Mumford’s algebraic Jacobian and the
Gau{\ss}--Cantor composition law, and develops the eigenvector map $\Psi_\mathcal C$ together with its inverse and its compatibilities
(shift, involutions, and time evolution).

\noindent Section~3 explains how periodic box--ball and tropical periodic Toda dynamics are recovered from the discrete periodic Toda flow,
emphasizing the role of $\mathbb{Q}(q)$ and valuations.

\noindent Section~4 proves the integrality result and derives $p$-adic consequences, and Section~5 collects questions and outlook.
Appendix~A relates our approach to the monodromy-matrix viewpoint.
Appendix~B reviews Cantor’s division polynomials (in the imaginary case) and explains how they can be used to compute multiples
of the translation element on $\mathrm{Jac}(\mathcal C)$ efficiently.
Appendix~C discusses a tropical analogue of Mumford’s Jacobian,
including a tropicalized inverse Abel--Jacobi map and the induced tropical linearization picture. \\ \\
As our paper relates two different areas of research (integrable systems and number theory) we tried to make the paper accessible
to researchers from both fields. Further, we have checked all computations with the help of SageMath.

\subsection{Acknowledgements} We wish to express our sincere gratitude to the (anonymous) referees for carefully reading the manuscript and for their valuable comments, which led to considerable improvements.
\newpage
\subsection{Notation and conventions.}
Our paper involves several layers of objects (Toda states, Jacobi matrices, spectral curves, Jacobians, and their tropical counterparts).
To keep symbols unambiguous we fix our conventions once and for all and collect the most frequently used notation in the following table.
\begin{table}[H]
\centering
\scriptsize
\setlength{\tabcolsep}{3.5pt}
\renewcommand{\arraystretch}{1.02}
\begin{tabular}{p{0.26\textwidth} p{0.66\textwidth} p{0.08\textwidth}}
\hline
Symbol & Meaning & First \\
\hline

\multicolumn{3}{l}{\textbf{General conventions}}\\
$\mathbb{N}$ & $\{0,1,2,\dots\}$ &  \\
$n$ & size of Toda system & \hyperref[sec-dptoda]{\S\ref*{sec-dptoda}} \\
$N$ & number of boxes in the periodic box--ball system & \hyperref[sec-motivation]{\S\ref*{sec-motivation}} \\
$R\subset K$ & base ring\footnotemark (char.~$0$) with fraction field $K$; key examples include fields: $\mathbb{R},\mathbb{Q}(q)$ and local rings: $\mathbb{Z}[q]_{(q)},\mathbb{Z}_p$ & \hyperref[sec-dptoda]{\S\ref*{sec-dptoda}} \\
\hline

\multicolumn{3}{l}{\textbf{Linear algebra}}\\
$M\in R^{n\times n}$ & matrix with entries in $R$ & \hyperref[sec-jacobi]{\S\ref*{sec-jacobi}} \\
$\mathcal M = M + xE_n$ & characteristic matrix of $M$ (with $E_n$ the unit matrix) & \hyperref[sec-spectral]{\S\ref*{sec-spectral}} \\
$\vert M \vert \in R$, $\vert \mathcal M \vert \in R[x]$ & determinant of $M$ and $\mathcal M$ & \hyperref[sec-spectral]{\S\ref*{sec-spectral}} \\
$M_{i,j} \in R^{(n-1)\times (n-1)}$& minor: delete the $i$-th row and $j$-th column & \hyperref[sec-spectral]{\S\ref*{sec-spectral}} \\
${}^i M _j \in R^{(n-i-j)\times(n-i-j)}$  & submatrix: delete first $i$ rows/cols and last $j$ rows/cols & \hyperref[sec-spectral]{\S\ref*{sec-spectral}} \\
$M_{j},\ {}^{i}\!M$ & abbreviations $M_j={}^{0}\!M_j$ and ${}^{i}\!M={}^{i}\!M_{0}$ & \hyperref[lemmakeymr]{Lem.~\ref*{lemmakeymr}} \\
$p_x$ & evaluation at $x=0$: $p_x=p(0)$ for $p\in R[x]$ & \hyperref[remark-hx]{Rem.~\ref*{remark-hx}} \\
$f|_{I_j=0}$, $f|_{V_j=0}$ & set $I_j=0$ in $f\in R[I_1,\dots,I_n,V_1,\dots,V_n][x]$ (similarly $f|_{V_j=0}$) & \hyperref[rellist]{Lem.~\ref*{rellist}} \\
$f_{I_j},\ f_{V_j}$ & shorthand for $f|_{I_j=0}$ and $f|_{V_j=0}$ & \hyperref[rellist]{Lem.~\ref*{rellist}} \\
\hline

\multicolumn{3}{l}{\textbf{Discrete periodic Toda}}\\
$(\bold I^t,\bold V^t)\in R^{2n}$ & Toda state at time $t$, with components $I_i^t,V_i^t$; indices are periodic $i = i+n$ & \hyperref[sec-dptoda]{\S\ref*{sec-dptoda}} \\
$\mathcal T : R^{2n}\dashrightarrow R^{2n}$ & Toda time flow/step: $(\bold I^{t+1},\bold V^{t+1}) = \mathcal T(\bold I^t,\bold V^t)$ & \hyperref[dpTflow]{(\ref*{dpTflow})} \\
$\sigma :R^{2n}\to R^{2n}$ & cyclic shift on indices $i\mapsto i+1$ (with $i = i+n$) and induced action: $(\bold I ^{t,\sigma},\bold V ^{t,\sigma}) = \sigma(\bold I ^t,\bold V ^t)$ & \hyperref[cyclicshift]{Def.~\ref*{cyclicshift}} \\
${}^\dagger :R^{2n}\to R^{2n}$ & dagger involution and induced action: $(\bold I ^{t,\dagger},\bold V ^{t,\dagger}) = (\bold I ^t,\bold V ^t)^\dagger$ & \hyperref[daggerinvolution]{Def.~\ref*{daggerinvolution}} \\
$z$ & invertible spectral parameter in $L_t$ & \hyperref[sec-jacobi]{\S\ref*{sec-jacobi}} \\
$L_t$, $\overline L_t$, $L^\sigma_t$, $L^\dagger_t \in R(z)^{n\times n}$ & periodic Jacobi/Lax matrices associated to \hspace{0cm}Toda states $(\bold I^t,\bold V^t)$, $(\bold I^{t+1},\bold V^{t+1})$, $(\bold I ^{t,\sigma},\bold V ^{t,\sigma}) $, $(\bold I^{t,\dagger}, \bold V^{t,\dagger})$ & \hyperref[sec-jacobi]{\S\ref*{sec-jacobi}} \\
$\mathcal L_t$, $\overline {\mathcal L}_t$, $\mathcal L_t^\sigma$, $\mathcal L_t^\dagger \in R(z)[x]^{n\times n }$ & characteristic matrices of $L_t$, $\overline L_t$, $L^\sigma_t$, $L_t^\dagger$ & \hyperref[sec-spectral]{\S\ref*{sec-spectral}} \\
$\sigma_{i,j}$, $s_i$, $t_i$ & involutions $R^{2n}\to R^{2n}$ on Toda states & \hyperref[involutions]{Def.~\ref*{involutions}} \\
\hline

\multicolumn{3}{l}{\textbf{Spectral curve and invariants}}\\
$\mathcal C = z \vert \mathcal L_t \vert$ 
& spectral curve $z^2+h(x)z+f=0$ of genus $g=n-1$ (Toda flow invariant)& \hyperref[sec-spectral]{\S\ref*{sec-spectral}} \\
$h(x)\in R[x],\ f\in R$ & coefficients defining $\mathcal C$ (Toda flow invariants)& \hyperref[sec-spectral]{\S\ref*{sec-spectral}} \\
$I_\Pi,\ V_\Pi$ & Toda flow invariants $I_\Pi=\prod_{i=1}^n I_i^t$, $V_\Pi=\prod_{i=1}^n V_i^t$ & \hyperref[sec-spectral]{\S\ref*{sec-spectral}} \\
$u^t$, $v^t$, $w^t \in R[x]$ & polynomials extracted from $\mathcal L_t$ & \hyperref[definitionuvw]{Def.~\ref*{definitionuvw}} \\
$\overline u^t$, $\overline v^t$, $\overline w^t \in R[x]$ & polynomials extracted from $\overline {\mathcal L}_t=\mathcal L_{t+1}$ & \hyperref[definitionuvw]{Def.~\ref*{definitionuvw}} \\
$u^{t,\sigma}$, $v^{t,\sigma}$, $w^{t,\sigma} \in R[x]$ & polynomials extracted from $\mathcal L_t^\sigma$ & \hyperref[definitionuvwsd]{Def.~\ref*{definitionuvwsd}} \\
$u^{t,\dagger}$, $v^{t,\dagger}$, $w^{t,\dagger} \in R[x]$ & polynomials extracted from $\mathcal L_t^\dagger$ & \hyperref[definitionuvwsd]{Def.~\ref*{definitionuvwsd}} \\
\hline

\multicolumn{3}{l}{\textbf{Linearization on Mumford Jacobian}}\\
$\jac$ & Jacobian of $\mathcal C$; elements written in balanced Mumford form $([P,Q],d)$ & \hyperref[sec-mumford]{\S\ref*{sec-mumford}} \\
$[P,Q]$ & Mumford representation (with $P\mid(Q^2+hQ+f)$ and $\deg P\le g$) & \hyperref[sec-mumford-repr]{\S\ref*{sec-mumford-repr}} \\
$\boxplus$ & addition on $\jac$ via Gau{\ss}--Cantor composition law & \hyperref[gausscomp]{\S\ref*{gausscomp}} \\
$\mathcal R _\mathcal C \subset R^{2n}$ & isolevel set of Toda states with fixed spectral curve $\mathcal C$ & \hyperref[isolevelset]{Def.~\ref*{isolevelset}} \\
$\Psi_\mathcal C: \mathcal R_\mathcal C\to\mathrm{Jac}(\mathcal C)$ & eigenvector (or algebraic Abel-Jacobi) map $(\bold I ^t,\bold V ^t)\mapsto([u^t,v^t],0)$ & \hyperref[defspectral]{Def.~\ref*{defspectral}} \\
$\mathfrak T\in\mathrm{Jac}(C)$ & translation element inducing Toda time flow on $\mathrm{Jac}(C)$ & \hyperref[lineartime]{Thm.~\ref*{lineartime}} \\
$\mathfrak{pa} \in \jac$ & Pell--Abel torsion point $[\infty_+-\infty_-]\in\mathrm{Jac}(\mathcal C)$, used to describe action of $\sigma$ and ${}^\dagger$ on $\jac$ & \hyperref[sec-mumford-examples]{\S\ref*{sec-mumford-examples}} \\
\hline

\multicolumn{3}{l}{\textbf{Tropical periodic Toda / periodic box--ball}}\\
$\mathcal B:\mathcal V_N\to \mathcal V_N$ & periodic box--ball flow on $\mathcal V_N\subset\{0,1\}^N$. & \hyperref[sec-motivation]{\S\ref*{sec-motivation}} \\
$\mathcal T^{\mathrm{trop}} : \mathbb R ^{2n} \dashrightarrow \mathbb R ^{2n} $ & tropical periodic Toda flow on $\mathbb R^{2n}$ (often restricted to $\mathbb N^{2n}$)& \hyperref[sec-trop-toda]{\S\ref*{sec-trop-toda}} \\
$\eta : \mathcal V_N \longrightarrow \bigsqcup _{n=1}^{\lfloor N/2 \rfloor}\mathbb N^{2n} $ & embedding from box--ball states into tropical Toda states. & \hyperref[sec-boxball-trop]{\S\ref*{sec-boxball-trop}} \\
$\mathfrak I_q:\mathbb N^{2n}\to\mathbb Q(q)^{2n}$ & (algebraic) lift from tropical to discrete periodic Toda & \hyperref[q-lift]{(\ref*{q-lift})} \\
$\mathcal T_p : \mathbb Z_p ^{2n} \dashrightarrow \mathbb Z_p ^{2n} $ & $p$-adic periodic Toda flow on $\mathbb Z_p ^{2n}$& \hyperref[sec-trop-toda]{\S\ref*{sec-trop-toda}} \\
\hline
\end{tabular}
\captionsetup{font=small}
\caption{List of most frequently used symbols and conventions.}
\end{table}
\vspace{-4pt}
\footnotetext{One can safely assume that $R$ is always one of the aforementioned rings.}
\noindent From Section~2 onward we often suppress the time superscript $t$ on $(\bold I ^t, \bold V^t)$ and on the associated objects, e.g., $L_t,u^t,v^t,w^t$, when no confusion can arise.
Whenever a statement compares different times (typically $t$ and $t+1$), we write the time superscripts explicitly or equivalently use $u = u^t$ and $\overline u=u^{t+1}$.

\section{Linearization of the discrete periodic Toda flow}
\subsection{The discrete periodic Toda flow}
\label{sec-dptoda}
The discrete periodic Toda flow over $R$ (for $n > 1$) is the rational map on Toda states \eqst{\mathcal T : R^{2n} \dashrightarrow R^{2n},} \eqst{ (\bold I ^t,\bold V^t) = (I_1^t,\dots,I_n^t,V_1^t,\dots,V_n^t) \longmapsto ( \bold I ^{t+1}, \bold V ^{t+1}),} defined, for $i \in \{1,..,n\}$ and $t \in \mathbb N$, by the rules 

\eq{
\label{dpTflow}
\begin{split}
I^{t+1}_i &= I_i^t+V_i^t-V^{t+1}_{i-1}, \\
V_i^{t+1} &= \frac{I^t_{i+1} V_i^t}{I_i^{t+1}},
\end{split}
}
with periodic boundary conditions
\eq{\label{periodicbc} I_{i+n}^t = I_i^t \text{ and } V_{i+n}^t = V_i^t.}
We always assume that \eqst{0 \neq \prod_{i=1}^n I^t_i \neq \prod_{i=1}^n V^t_i \neq 0 .}
Under the above assumption, one can show (cf., \cite{tokihiro1})
\eqst{ I^{t+1}_i = I^{t}_{i+1}\tfrac{I_i^t I_{i-1}^t\cdots I^t_{i+2} + V^t_i I^t_{i-1}\cdots I^t_{i+2}+\dots+V^t_iV^t_{i-1}\cdots V^t_{i+2}}{I_{i-1}^tI_{i-2}^t\cdots I_{i+1}^t + V^t_{i-1}I^t_{i-2}\cdots I^t_{i+1}+\dots+V^t_{i-1} V^t_{i-2}\cdots V^t_{i+1}}} and \eqst{V^{t+1}_i = V_i^t \tfrac{I_{i-1}^tI_{i-2}^t\cdots I_{i+1}^t + V^t_{i-1}I^t_{i-2}\cdots I^t_{i+1}+\dots+V^t_{i-1} V^t_{i-2}\cdots V^t_{i+1}}{I_i^t I_{i-1}^t\cdots I^t_{i+2} + V^t_i I^t_{i-1}\cdots I^t_{i+2}+\dots+V^t_iV^t_{i-1}\cdots V^t_{i+2}}.} 
\begin{remark}
This also follows easily from our new inverse algebraic Abel-Jacobi map, using Proposition \ref{inversepsi}.
\end{remark}

\subsection{Jacobi matrices}
\label{sec-jacobi}
A more elegant and natural description (cf., \cite{spectrumjacobi,tokihiro1,iwao}) of the phase space $R^{2n}$ of the discrete periodic Toda flow is given by the space of periodic Jacobi matrices \eqst{\mathrm{Jac}_n^{\mathrm{per}}(R) = \left\{L_t=\begin{bmatrix}I_1^t + V_n^t  & 1&  & (-1)^{n-1} I_n^t V_n^t / z \\ I_1^t V_1^t& I_2^t + V_1^t & \ddots & \\ & \ddots & \ddots & 1 \\ (-1)^{n-1}z & &I_{n-1}^t V_{n-1}^t &I_n^t+V_{n-1}^t \end{bmatrix}\right\},} where $z$ is an invertible spectral parameter\footnote{The reason for our chosen sign normalization will become clear later on.}. \\
Of central importance is the decomposition \eqst{L_t = M_tR_t,}
with the matrices \eqst{M_t = \begin{bmatrix} 1 &  & & (-1)^{n-1} V_n^t / z \\ V_1^t & 1 & & \\ & \ddots & \ddots & \\ & & V_{n-1}^t & 1 \end{bmatrix}} and \eqst{R_t = \begin{bmatrix} I_1^t & 1 & & \\ & I_2^t & \ddots & \\ & & \ddots & 1 \\\ (-1)^{n-1}z & & & I_n^t \end{bmatrix}.} On the phase space $\mathrm{Jac}_n^{\mathrm{per}}(R)$ the discrete periodic Toda flow is equivalent to the fundamental matrix equation \eqst{ L_{t+1} = R_t M_t} or, equivalently, in isospectral form \eq{\label{isospec} {L}_{t+1}  = R_t L_t R_t^{-1}.}
It will be important to keep in mind the explicit shape of $L_t$ 
and \eqst{\overline L _{t} = {L}_{t+1} = \begin{bmatrix} I_1^t + V_1^t  & 1&  & (-1)^{n-1}I_1^t V_n^t / z \\ I_2^t V_1^t& I_2^t + V_2^t & \ddots & \\ & \ddots & \ddots & 1 \\ (-1)^{n-1}z & &I_{n}^t V_{n-1}^t &I_n^t+V_{n}^t \end{bmatrix}.} 

\begin{remark}
From now on we will usually suppress the time parameter $t$. Unless stated otherwise, $I_i$ and $V_j$ always mean $I_i^t$ and $V_j^t$ and so on.
\end{remark}
\begin{remark}
We will use the coordinates $(\bold{I},\bold{V})$ and the corresponding matrix $L$ interchangeably. 
\end{remark}
\subsection{The spectral curve}
\label{sec-spectral}
With $L \in \mathrm{Jac}_n^{\mathrm{per}}(R)$ 
we associate the spectral curve $\mathcal C = \mathcal C_L$ 
defined by the (normalized) characteristic polynomial \eqst{\mathcal C : z \, \vert \mathcal L \vert =  z \, \vert L_t + x E_n \vert = z^2 + h_n(x) z + f_n(x) = 0.} From (\ref{isospec}) we see that the spectral curve is independent of the time flow, i.e., the polynomials $h(x) = h_n(x)$ and $f(x)=f_n(x)$ are invariant under the time flow. Moreover, it directly follows that the quantities \eqst{I_{\Pi} = \prod_{i=1}^n I_i^t \ \ \  \text{ and } \ \ \  V_{\Pi} = \prod_{i=1}^n V_i^t} are also invariant under the time flow. \begin{remark}We always assume that the spectral curve $\mathcal C$ is non-degenerate. 
\end{remark}
\noindent By expanding the last row of the determinant $\vert \mathcal L \vert$ we find \eqst{\vert \mathcal L \vert = (I_n+V_{n-1}+x)\vert \mathcal L_1 \vert-I_{n-1}V_{n-1}\vert\mathcal L _{n-1,n}\vert + z \vert \mathcal L_{1,n}\vert.} Expanding further \eqst{\vert\mathcal L_{n-1,n}\vert = \vert \mathcal L_2 \vert - (I_nV_n/z) (I_1V_1 \cdots I_{n-2}V_{n-2})} and \eqst{\vert\mathcal L _{1,n}\vert = 1-(I_nV_n/z) \vert {}^{1}\mathcal L_1 \vert,} we find: \eqst{h(x) = (I_n+V_{n-1} +x) \vert \mathcal L_1 \vert - I_n V_n \vert ^1 \mathcal L _1 \vert - I_{n-1} V_{n-1} \vert \mathcal L _ 2 \vert  \in R[x]} and \eqst{f = \prod_{i=1}^n I_i V_i \in R.}  
\begin{remark}
\label{remark-hx}
Observe that $h_x = h(0) = I_{\Pi}+V_{\Pi}$ .
\end{remark}
\noindent The following definition will be central to the present work.
\begin{definition}
\label{definitionuvw}
Define the polynomials $u,v,w,\overline u, \overline v, \overline w \in R[x]$ by 

\[
\begin{aligned}
&u  = \vert \mathcal L _1 \vert, 
&v  =  I^t_n V^t_n \vert {}^1\mathcal L _1 \vert, \ \ \ 
&w = I^t_{n-1}V^t_{n-1} \vert \mathcal L _ 2 \vert \\
&\overline u = \vert \overline{\mathcal L} _1 \vert, 
&\overline v =   I^t_1 V^t_n \vert {}^1\overline{\mathcal L} _1 \vert , \ \ \ 
&\overline w = I^t_{n}V^t_{n-1} \vert \overline{ \mathcal L} _ 2 \vert .
\end{aligned}
\]
\end{definition}
\noindent Now we can write \eq{\label{definitionh}\text{{$h = (I_n+V_{n-1}+x)u-v-w = (I_n+V_n+x)\overline u - \overline v - \overline w = \overline h $}}.}

\noindent The following simple facts will be used throughout the paper.
\begin{lemma} 
\label{lemmakeymr} 
Let $k,l \in \mathbb N$ and assume $l+k \geq 1$. Then we have 
\begin{align*}
\vert{}^k\mathcal L _l\vert &= (I_{k+1}+V_k +x)\vert{}^{k+1}\mathcal L _l\vert-I_{k+1}V_{k+1} \vert{}^{k+2}\mathcal L_l\vert, \\
\vert{}^k\mathcal L _l\vert &= (I_{n-l}+V_{n-l-1} +x)\vert{}^{k}\mathcal L _{l+1}\vert-I_{n-l-1}V_{n-l-1} \vert{}^{k}\mathcal L_{l+2}\vert,\\
\vert{}^k\overline{\mathcal {L}} _l\vert &= (I_{k+1}+V_{k+1} +x)\vert{}^{k+1}\overline{\mathcal {L}} _l\vert-I_{k+2}V_{k+1} \vert{}^{k+2}\overline{\mathcal {L}}_l\vert\\
\vert{}^k\overline{\mathcal {L}} _l\vert &= (I_{n-l}+V_{n-l} +x)\vert{}^{k}\overline{\mathcal {L}}_{l+1}\vert-I_{n-l}V_{n-l-1} \vert{}^{k}\overline{\mathcal {L}}_{l+2}\vert.
\end{align*}
\end{lemma}
\begin{proof}
Those are simply the expansions formulas for determinants of tridiagonal matrices for the first and last row.
\end{proof}

\subsubsection{Different models of $\mathcal C$}
\noindent Hyperelliptic curves $\mathcal C$ arising from the discrete periodic Toda flow are real hyperelliptic curves, i.e., they have two points at infinity, which we denote by $\infty^+$ and $\infty^-$. 

\subsubsection{Standard model $\widetilde {\mathcal C}$}
The standard model $\widetilde{\mathcal C}$ is defined by \eqst{\tilde {\mathcal C} : y^2 = F(x) = h(x)^2-4 f } and is obtained from $\mathcal C$ by the birational transformation \eqst{(x,z) \longmapsto (x,2 z + h(x)) = (x,y).}
In particular, we see that $\mathcal C$ has genus \eqst{g = \mathrm{deg}(h(x))-1= n-1.} 
\subsubsection{Imaginary model $\mathcal C _\alpha$}
Passing from a real model $\widetilde{\mathcal C}$ to an imaginary model ${\mathcal C} _\alpha$ depends on the choice $\alpha$ of a root\footnote{We always assume that $F(x)$ has pairwise distinct roots. Generically this is true.} of $F(x) = \prod_{i=1}^{2g+2} (x-\alpha_i)$, and is defined over an extension of $R$. Given such a root $\alpha$, set $\beta^2 = \prod_{i=1,\alpha_i \neq \alpha}^{2g+2} (\alpha -\alpha_i)$, then the birational transformation \eqst{(x,y) \longmapsto (\tfrac{1}{x-\alpha},\tfrac{y}{\beta (x-\alpha)^{g+1}}) = (x_\alpha, y_\alpha)} 
defines the imaginary model \eqst{{\mathcal C}_\alpha :  y_\alpha ^2 = F_\alpha(x_\alpha)} with $F_\alpha$ monic and $\mathrm{deg}(F_\alpha) = 2g +1$. The curve $\mathcal C _\alpha$ has exactly one point at infinity.
\subsection{The cyclic shift and some involutions}

\begin{definition}
\label{cyclicshift}
The cyclic shift map $\sigma$ is defined by \eqst{\sigma: (I_i,V_i) \longmapsto (I_{i+1},V_{i+1}),} with the usual convention (\ref{periodicbc}). We write $(\bold I, \bold V) \mapsto (\bold I ^\sigma, \bold V ^\sigma)$.
\end{definition}
\begin{lemma}
\label{baslem}
The spectral curve $\mathcal C$ is invariant under the cyclic shift map $\sigma$, i.e.,  \eqst{h^\sigma = h \text{ and }f^\sigma = f.}
\end{lemma}
\begin{proof}
We need to show that \eqst{h = h^\sigma = (I_1+V_n+x)\vert \mathcal L^\sigma_1\vert-I_1V_1 \vert{}^1\mathcal L^\sigma_1\vert - I_nV_n \vert \mathcal L^\sigma_2\vert.}
This follows from $\vert \mathcal L^\sigma_1\vert = \vert {^1\mathcal L \vert}$, $\vert {}^1 \mathcal L^\sigma_1 \vert = \vert{}^2 \mathcal L \vert$ and $\vert \mathcal L^\sigma_2\vert = \vert {}^1 \mathcal L _1 \vert$ in combination with Lemma \ref{lemmakeymr}.
\end{proof}
\begin{definition}
\label{daggerinvolution}
The dagger involution ${}^\dagger :R^{2n}\longrightarrow R^{2n}$ is defined by \eqst{{}^\dagger : (I_1,\dots,I_n,V_1,\dots,V_n) \mapsto (V_{n-1},\dots,V_1,V_n,I_{n-1},\dots,I_1,I_n).}
We write $(\bold I,\bold V) \mapsto (\bold I ^\dagger,\bold V ^\dagger)$.
\end{definition}
\begin{lemma}
The spectral curve is invariant under the dagger involution ${}^\dagger$, i.e.,  \eqst{h^\dagger = h \text{ and }f^\dagger = f.}
\end{lemma}
\begin{proof}
This follows from Lemma \ref{rellist} in combination with Lemma \ref{baslem}
\end{proof}

\begin{definition}
\label{definitionuvwsd}
Define the polynomials $u^\sigma,v^\sigma,w^\sigma,u^\dagger,v^\dagger,w^\dagger \in R[x]$ by
\[
\begin{aligned}
&u^\sigma  = \vert \mathcal L^\sigma _1 \vert, 
&v^\sigma =  I^t_n V^t_n \vert {}^1\mathcal L^\sigma _1 \vert, \ \ \ 
&w^\sigma =  I^t_{n-1}V^t_{n-1} \vert \mathcal L^\sigma _ 2 \vert \\
&u^\dagger = \vert \mathcal L ^\dagger _1 \vert, 
&v^\dagger = I^t_n V^t_n \vert {}^1\mathcal L^\dagger _1 \vert , \ \ \ 
&w^\dagger =  I^t_{1}V^t_{1} \vert \mathcal L^\dagger _ 2 \vert .
\end{aligned}
\]
\end{definition}

\noindent We will also need the following involutions
\begin{definition}
\label{involutions}
For $i, j \in \{1,\dots,n\}$, with $i < j$, we define the involution \eqst{\sigma_{i,j}:I_i \leftrightarrow I_j, V_i \leftrightarrow V_j.}
Further we define the involutions \eqst{s_i : I_j \mapsto \left\{ \begin{array}{c}  -I_i\text{, if $j=i$} \\  \  I_j \text{, if $j\neq i$}  \end{array} \right.} and 
\eqst{t_i : V_j \mapsto \left\{ \begin{array}{c}  -V_i \text{, if $j=i$} \\  \  V_j \text{, if $j\neq i$}  \end{array} .\right.} 

\end{definition}

\subsection{Some useful observations}

\begin{lemma}
\label{rellist}
The following relations hold

\[
\begin{aligned}
u&= \tfrac{h_{I_n,V_{n-1}}}{x}, & v &= I_n V_n \tfrac{u_{I_1,V_n}}{ x} , & w &= I_{n-1} V_{n-1} \tfrac{u_{I_{n-1},V_{n-2}}}{ x},  \\ 
\overline u &= \tfrac{h_{I_n,V_n}}{x},  &\overline{v} &= I_1 V_n \tfrac{\overline u _{I_1,V_1}}{ x} , & \overline w &= I_n V_{n-1} \tfrac{\overline u _{I_{n-1},V_{n-1}}}{ x}, \\ 
u^\sigma &= \tfrac{h_{I_1,V_n}}{ x}  , &v^\sigma &= I_1 V_1 \tfrac{u^\sigma_{I_2,V_1}}x , & w^\sigma &= I_n V_n \tfrac{u^\sigma_{I_n,V_{n-1}}}x, \\ 
u^\dagger &= u^\sigma,  & {v^\dagger} &= w^\sigma, &  w^\dagger &= v^\sigma.
\end{aligned}
\]

\end{lemma}
\begin{proof}
We will only prove the first two relations. The other relations follow similarly. The first follows trivially from the definition of $h$, see (\ref{definitionh}). The second one follows from the expansion \eqst{\vert \mathcal L _1 \vert = (I_1+V_n+x) \vert {}^1 \mathcal L _1 \vert - I_1V_1 \vert {}^2 \mathcal L _1 \vert.}
\end{proof}
\begin{corollary}
\label{easycor}
We have \eqst{w^\sigma = v.}
\end{corollary}
\begin{remark}
Setting $I_i = 0$ above Lemma can be expressed in terms of the involution $s_i$, namely we have \eqst{u_{I_i} = \tfrac 1 2 (u+u^{s_i}).} The same for $V_i$ and $t_i$. The reason is simply that each variable $I_i$ and $V_i$ appears at most linearly in $u$.
\end{remark}

\begin{lemma}
\label{subfree}
The polynomials \eqst{h,u,v,w,\overline u, \overline v, \overline w,u^\sigma,v^\sigma,w^\sigma,u^\dagger,v^\dagger,w^\dagger \in R[x]} are all subtraction-free.
\end{lemma}
\begin{proof}
From Lemma \ref{rellist} we see that it is enough to show that $h$ is subtraction-free. From the formula \begin{align*}h &= (I_n+V_{n-1}+x) u - v - w \\ &=x u + I_n u_{V_n}+V_{n-1}u_{I_{n-1}} \end{align*} we only need to show that $u = \vert \mathcal L _1 \vert$ is subtraction-free. This follows from the tridiagonal structure of $\vert \mathcal L _1 \vert$ (and Lemma \ref{lemmakeymr}), as one can immediately see that the only negative terms appearing $\vert \mathcal L _ 1 \vert$ are of the form $-I_j V_j \vert \mathcal L _{n-j+1} \vert $ and those are cancelled by the expressions $(I_j +V_{j-1}+ x) (I_{j+1}+V_j+x) \vert \mathcal L _{n-j+1} \vert$.
\end{proof}
\begin{remark}
The last Lemma is the main reason for our choice of normalization of $\mathcal C$, as subtraction-free polynomials can be nicely tropicalized. This will be used in appendix \ref{tropicalmumford}.
\end{remark}
\noindent We now come to the important 
\begin{lemma}
\label{basicrel}
We have \eqst{vw - f = I_{n-1} \,I_n \,V_{n-1} \,V_n \, u \, \vert^{1}\mathcal L _2 \vert.}
\end{lemma}
\begin{proof}
The statement is equivalent to \eqst{\vert {}^1\mathcal L_1\vert \vert \mathcal L _2 \vert - \vert \mathcal L _1\vert \vert{}^1\mathcal L _2\vert = I_1V_1 \cdots I_{n-2}V_{n-2}.} Applying Lemma \ref{lemmakeymr} to $\vert \mathcal L _2 \vert$ and $\vert \mathcal L _1\vert$ we obtain \eqst{\vert{}^2 \mathcal L _1\vert \vert{}^1 \mathcal L _2\vert - \vert{}^1 \mathcal L _1\vert \vert{}^2 \mathcal L _2\vert = I_2 V_2 \cdots I_{n-2}V_{n-2}.} The statement now follows from iterative applications of Lemma \ref{lemmakeymr} to the two factors with the lowest upper indices and the trivial identity \eqst{\vert{}^{n-2}\mathcal L_1\vert\vert{}^{n-3}\mathcal L_2\vert-\vert{}^{n-3}\mathcal L_1\vert\vert{}^{n-2}\mathcal L_2\vert = I_{n-2}V_{n-2}.}
\end{proof}

\begin{lemma}
We have\eqst{I_n V_n \vert {}^1\mathcal L _2 \vert = (I_n+V_{n-1}+x)v^{\sigma^{-1}} -u^{\sigma^{-1}}} and
\eqst{\vert{}^1\mathcal L _2\vert = \tfrac{u_{I_1,V_{n-1},V_n}}{x}.}
\end{lemma}
\begin{proof}
This follows from $\vert {}^1 \mathcal L _ 2 \vert = \vert {}^2 \mathcal L^{\sigma^{-1}} _ 1 \vert$ and Lemma \ref{lemmakeymr}.
\end{proof}
\begin{lemma}
\label{keyrelation}
The following relations hold
\begin{align*}I_n u - v &= I_n u_{V_{n}} \\ &= I_n \overline u - \overline w, \\ (V_{n-1}+x)u-w &= (V_n + x)\overline u - \overline v.\end{align*}
\end{lemma}
\begin{proof}
The first statement follows immediately from the expansions \begin{align*}u &= (I_1+V_n+x)\vert {}^1 \mathcal L_1 \vert - I_1V_1\vert{}^2\mathcal L _1 \vert, \\ \overline u &= (I_1+V_1+x) \vert{}^1\overline{\mathcal L}_1\vert-I_2V_1 \vert{}^2\overline{\mathcal L}_1\vert\end{align*} and Lemma \ref{rellist}. The second assertion follows from the first and $h = \overline h.$
\end{proof}

\begin{lemma}
\label{lemmarecursion}
We have the following recursion relations \begin{align*}
u_n &= (u_n)_{V_n}+V_n u_{n-1}^{\sigma_{n-1}}, \\
v_{n} &= I_{n}V_{n} u_{n-1}^{\sigma_{n-1}}, \\
w_{n} &= I_{n-1}V_{n-1}u_{n-1}^{\sigma_{n-1,n}}. 
\end{align*}

\end{lemma}
\begin{proof}
We start with the second relation. This follows from \eqst{ \vert {}^1\mathcal L _1 \vert = \vert\mathcal L^{\sigma_n}_2\vert = \vert (\mathcal L_{n-1}^{\sigma_{n-1}})_1 \vert.} The first relation now follows from \eqst{u_n - (u_n)_{V_n} = u_n - u_n + V_n \vert{}^1\mathcal L_1 \vert} and the second relation. The other relations follow similarly.
\end{proof}
\begin{remark}
In particular, we have \eqst{(I_n+V_{n-1}+x) \, u_n = h_n +I_nV_n u_{n-1}^{\sigma_{n-1}}+I_{n-1}V_{n-1} u_{n-1}^{\sigma_{n-1,n}}.}
\end{remark}
\begin{remark}
We have \eqst{(\overline u _ n)_{V_{n-1}} = (u_n)_{V_n} .}
\end{remark}

\subsection{Mumford's algebraic description of $\jac$}
\label{sec-mumford}
Mumford's beautiful algebraic description\footnote{See \cite{mumfordtheta} for Mumford's original treatment in the case of imaginary hyperelliptic curves.} of $\jac$, for a real hyperelliptic curve $\mathcal C : z^2 +h(x) z + f = 0$ of genus $g$ as above, is explained in (varying degrees of) detail in \cite{galbraith}, \cite{comparison} and \cite{wood}. We will mostly follow \cite{galbraith,comparison}. 
\subsubsection{Mumford's representation}
\label{sec-mumford-repr}
Note any reduced divisor $D'$ not supported at infinity can be represented by two polynomials $P(x),Q(x) \in R[x]$, such that $P$ is monic, $\mathrm{deg}(P) \leq g$ and \eqst{P \ \vert \ Q^2 + h Q + f.} 
\begin{remark}The last condition assures that for $\alpha \in \overline K$ with $P(\alpha) = 0$ we have \eqst{(\alpha,Q(\alpha)) \in \mathcal C.}
\end{remark} 
\noindent The unique Mumford representation of $D'$ is defined by \eqst{D' = [P,Q \text{ mod }P] .} 
\subsubsection{All of $\jac$}
We introduce the divisor $D_\infty = \lceil \frac g 2 \rceil \infty ^+ + \lfloor \frac g 2 \rfloor \infty^-$. Following \cite{galbraith}, any element of $\jac$ can be uniquely written as difference $D - D_\infty,$ with \eqst{D = D' + d \infty^+ + (g-d-\mathrm{deg}(D'))\infty^-,} where $D'$ is a Mumford representation and $0 \leq d \leq g-\mathrm{deg}(D')$ is a natural number. This is the so-called balanced (Mumford) representation. Using the above notation, we obtain the following description \eqst{\jac = \{ ([P,Q],d) \}.} When the balancing parameter $d$ is understood, we simply write $[P,Q] = ([P,Q],d) \in \jac$.
\subsubsection{Examples}
\label{sec-mumford-examples}
The zero element is given by \eqst{[0] = ([1,0],\lceil \tfrac g 2 \rceil) \in \jac.} The torsion points \eqst{\mathfrak{pa} = [\infty^+ - \infty^-] = ([1,0],\lceil \tfrac g 2 \rceil + 1) \in \jac,} \eqst{-\mathfrak {pa} = [\infty^- - \infty^+] = ([1,0],\lceil \tfrac g 2 \rceil - 1) \in \jac} will appear prominently in the following. The element $\mathfrak {pa}$ corresponds to a fundamental solution of the famous Pell--Abel equation for $\mathcal C$, as explained in \cite{gendron}. Our framework gives a natural proof of the fact that $\mathfrak {pa} \in \jac[n]$, see Theorem \ref{torsionpoint}.

\subsection{Gauß composition and addition law in $\jac$}
\label{gausscomp}
Following \cite{galbraith} we will explain the addition law \eqst{([P_1,Q_1],d_1)\boxplus ([P_2,Q_2],d_2) = ([P,Q],d) \in \jac} in our set-up. The main ingredient is Cantor's\footnote{Originally, Cantor worked over imaginary hyperelliptic curves where all subtleties at infinity disappear, in our set-up of real hyperelliptic curves some care is initially needed.} adaptation \cite{cantorjacobian} of Gauß composition law for quadratic forms (over $\mathbb Z$) to hyperelliptic curves (see also \cite{wood}), which consists of a composition step and reduction step for the Mumford representations. We extract the part of the general algorithm for real hyperelliptic curves from \cite{galbraith} necessary for our set-up as follows.

\subsubsection*{I. Composition}
Compute polynomials $s \in R[x]$ (monic) and $f_1,f_2,f_3 \in R[x]$ such that \eqst{s = \mathrm{gcd}(P_1,P_2,Q_1+Q_2+h) = f_1 P_1 + f_2 P_2 + f_3 (Q_1+Q_2+h).}
Then the composition step $[\tilde P,\tilde Q]$ 
is defined by \eqst{ [\tilde P, \tilde Q] = [P_1 P_2 / s^2, (f_1 P_1 Q_2 + f_2 P_2 Q_1 + f_3(Q_1 Q_2 -f))/s \text{ mod } \tilde P].}
\subsubsection*{II. Reduction}
The reduction step $[P,Q] = [P_1,Q_1]\boxplus[P_2,Q_2] $ is defined by \eqst{[P,Q] = [(\tilde Q ^2+h \tilde Q + f )/\tilde P \text{ made monic}, -(\tilde Q + h) \text{ mod } P ].}
\subsubsection*{III. Composition at infinity}
We will only look at sums of the form $([u,v],0) \boxplus \mathfrak a$, where $\mathfrak a \in \{\mathfrak{pa},\mathfrak T\}$ (see (\ref{divT}) for the definition of the latter). In those two cases, the set-up works such that the result is of the form $([\overline u,\overline v],0)$ where $[\overline u ,\overline v]$ is given by the composition of the Mumford representations.
\subsubsection{Example of inverse element}
The inverse of a balanced Mumford representation is given in our set-up as follows. If $g$ is even, we have \eqst{-([P,Q],d) = ([P,-(h+Q) \text{ mod } P],g-\mathrm{deg}(P)-d).} If $g$ is odd, we have \eqst{-([P,Q],d) =  \left\{ \begin{array}{c} ([P,-(h+Q) \text{ mod } P],g-d-\mathrm{deg}(P)+1) \text{, if $d>0$} \\  \  ([P,-(h+Q) \text{ mod } P],0)\boxminus \mathfrak{pa} \text{, if $d=0$}  \end{array} . \right.}

\subsection{The eigenvector map $\Psi$ and its inverse}
The idea of the eigenvector map goes back at least to \cite{vmm}. We fix a spectral curve $\mathcal C : z \vert\mathcal L\vert =  z^2 +h(x) z + f = 0$ as above and consider the vector \eqst{\phi = \begin{bmatrix}\phi_1(x,z) \\ \vdots \\ \phi_n(x,z)\end{bmatrix},} defined by \eqst{\phi_i(x,z) = (-1)^{i} \vert \mathcal L _{i,n} \vert.} A small calculation 
shows \eqst{\mathcal L \, \phi = \begin{bmatrix} 0 \\ \vdots \\ 0 \\ (-1)^{n} \vert \mathcal L \vert\end{bmatrix} ,}i.e., on the spectral curve $\mathcal C$ the vector $\phi$ is an eigenvector of $L_t$ with eigenvalue $-x$. In particular 
we immediately see \eq{\label{equ} \phi_n(x,z) = 0 \iff u(x)=0,} \eq{\label{eqv} \phi_1(x,z) = 0 \iff v(x)  = z,}  \eqst{\phi_{n-1}(x,z) = 0 \iff w(x) = f/z.}
Before we define the eigenvector map $\Psi_\mathcal C$, we need the following
\begin{definition}
\label{isolevelset}
Let $\mathcal C$ be a spectral curve as above. The isolevel set $\mathcal R _\mathcal C$ is defined by \eqst{\mathcal R _{\mathcal C} = \{ (\bold{I},\bold{V}) \in R^{2n} \ \vert \ \mathcal C_{(\bold{I},\bold{V})} = \mathcal C \}.}  
\end{definition}
\noindent Now we can finally make the fundamental
\begin{definition}
\label{defspectral}
Let $\mathcal C$ be a spectral curve as above. The eigenvector map \eqst{\Psi_\mathcal C:\mathcal R _\mathcal C
\longrightarrow \jac} is defined by \eqst{(\bold I,\bold V)
\longmapsto ([u,v],0).}
\end{definition}

\begin{lemma} The eigenvector map $\Psi_\mathcal C$ is well-defined and injective.
\end{lemma}
\begin{proof}
The degree of $u$ is equal to $g = n-1$ and the degree of $v$ is equal to $g-1$. The condition \eqst{u \, \vert \, v^2+vh+f} follows from (\ref{eqv}) and (\ref{equ}) or we can use Lemma \ref{basicrel} (and the definition of $h$) to see \eqst{u \ \vert \ v^2+vh+f = u((I_n+V_{n-1}+x) v - I_{n-1}I_nV_{n-1}V_n \vert^{1}\mathcal L _2 \vert).}
The injectivity follows from the next proposition where we construct a (left) inverse of $\Psi$. 
\end{proof}
\noindent But first, we have the following key
\begin{lemma}
\label{invlem}
We have \eqst{I_n = \frac{v_x +I_{\Pi}}{u_x} \ \ \ \text{ and } \ \ \ V_{n-1} = \frac{w_x + V_{\Pi}}{u_x}.}
\end{lemma}
\begin{proof}
We need to show \eqst{u_x = \vert L _ 1 \vert = I_1 \cdots I_{n-1} + V_n \vert {}^1 L_1\vert.} Applying Lemma \ref{lemmakeymr} modulo $x$ to $\vert L _1 \vert$, we see \eqst{u_x = (I_1+V_n) \vert {}^1 L _1\vert - I_1 V_1 \vert {}^2 L _1\vert = I_1(\vert {}^1 L_1 \vert - V_1 \vert {}^2 L_1\vert) + V_n \vert {}^1 L_1\vert.} From iterative applications of Lemma \ref{lemmakeymr} modulo $x$ we immediately obtain \eqst{\vert{}^1L_1\vert - V_1\vert{}^2L_1\vert = I_2 \cdots I_{n-1},} thus finishing the proof of the first assertion. \\
The second assertion follows from the first together with \eqst{h_x = (I_n+V_{n-1}) u_x - v_x - w_x} and the observation $h_x = I_{\Pi}+V_{\Pi}$.
\end{proof}
\begin{proposition}[Algebraic inverse Abel-Jacobi map]
\label{inversepsi}
The eigenvector map $\Psi_\mathcal C$ has an explicit left inverse denoted by \eqst{\Psi^{-1}_\mathcal C : \mathrm{Im}(\Psi_\mathcal C)\longrightarrow \mathcal R_\mathcal C.}
\end{proposition}
\begin{proof}
Applying the cyclic shift map $\sigma^k$, for $k \in \{1,..,n\},$ to Lemma \ref{invlem} we get \eqst{I_{k} = \tfrac{\sigma^k(v)_x+I_{\Pi}}{\sigma^k(u)_x}} and\footnote{Here $\mathrm{l.c.}$ means leading coefficient.} \eqst{V_{k} = \tfrac{\mathrm{l.c.}(\sigma^k(v))}{I_k}.}
This defines our map $\Psi^{-1}_\mathcal C: [u,v] \longmapsto (\bold I,\bold V)$ and by construction we have $\Psi^{-1}_\mathcal C \circ \Psi_\mathcal C = \mathrm{Id}_{\mathcal R_\mathcal C}$. 
\end{proof}
\begin{remark}
We usually write $\Psi = \Psi_\mathcal C$.
\end{remark}
\subsubsection{Further properties of $\Psi$}
The compatibility between the eigenvector map $\Psi$ and the cyclic shift $\sigma$ is given by the following important
\begin{theorem}
\label{torsionpoint}
We have \eqst{[u^\sigma,v^\sigma] = [u,v] \boxplus \mathfrak{pa} \in \jac.}
In particular this shows \eqst{\mathfrak{pa} \in \mathrm{Jac}(\mathcal C)[n].}
\end{theorem}
\begin{proof}
Following the recipe of the Gauß composition law, we immediately see that we need to show that 
\eqst{[u^\sigma,v^\sigma] = [(v^2+v h +f) /(I_n V_n u),(I_1+V_n+x) u^\sigma-(h+v)].}
For the first equality, we recall \eqst{x u^\sigma = h_{I_1,V_n} = (I_n+V_{n-1}+x)u_{I_1,V_n}-w_{I_1,V_n}.}
But from \eqst{u = (I_1+V_n+x)\vert {}^1\mathcal L _1\vert - I_1V_1 \vert{}^2\mathcal L_1\vert} and \eqst{w = (I_1+V_n+x)\vert {}^1\mathcal L _2\vert - I_1V_1 \vert{}^2\mathcal L_2\vert} we obtain \eqst{u^\sigma = (I_n+V_{n-1}+x) \vert{}^1\mathcal L _1 \vert - I_{n-1} V_{n-1} \vert{}^1 \mathcal L _2 \vert.}
On the other hand, using (\ref{definitionh}) and Lemma \ref{basicrel} we obtain \eqst{(v^2+vh+f)/(I_nV_n u) = (I_n+V_{n-1}+x) \vert{}^1 \mathcal L _1\vert - I_{n-1}V_{n-1}\vert {}^1 \mathcal L _ 2\vert.}
For the second equality we use $h^\sigma = h$ and $w^\sigma = v$.
\end{proof}

\begin{lemma}
We have \eqst{-[u,v] =  \left\{ \begin{array}{c} [u,w] \text{ ,if $g$ even} \\  \  [u,w] \boxminus \mathfrak{pa} \text{, if $g$ odd}  \end{array} . \right.} 
\end{lemma}

\begin{lemma}
We have \eqst{[u,w] \boxplus \mathfrak{pa} = [\sigma^{-1}(u),\sigma^{-1}(w)].}
\end{lemma}
\begin{proof}
This comes from \eqst{[u,w]\boxplus \mathfrak {pa} = \boxminus [u,v] \boxplus \mathfrak{pa} = \boxminus([u,v]\boxminus \mathfrak{pa}) = \boxminus(\sigma^{-1}[u,v]) = \sigma^{-1}[u,w].}
\end{proof}
\noindent Further, we have 
\begin{lemma}
We have \eqst{[u^\dagger,v^\dagger] = {[u,w]} \boxminus \mathfrak{pa}.}
\end{lemma}
\begin{proof}
This follows from Lemma \ref{rellist} and the last Lemma.
\end{proof}

\subsection{Linearization of the discrete periodic Toda flow}
One of our main results is the following
\begin{theorem}
\label{lineartime}
Define the divisor\eq{\label{divT}\mathfrak T = ([x,-I_{\Pi}],\lceil \tfrac g 2 \rceil - 1) \in \mathrm{Jac}(\mathcal C).} Then we have \eqst{[\overline u,\overline v] = [u,v] \boxplus \mathfrak T \in \jac,} i.e., the following diagram is commutative

\begin{equation*}
\begin{tikzcd}
R^{2n} \arrow[r,"\Psi"] \arrow[d,"\mathcal T "'] &
\mathrm{Jac}(\mathcal C)  \arrow[d," \cdot  \boxplus \mathfrak T "]
\\
R^{2n} \arrow[r,"\Psi"]  &
\mathrm{Jac}(\mathcal C). 
\end{tikzcd}
\end{equation*}
\end{theorem}
\begin{proof}
Following the composition step of the Gauss composition law and using Lemma \ref{invlem}, we find that \eqst{[\tilde u, \tilde v] = [x u,-I_n u + v].} Performing the reduction step, we need to show \eqst{[\overline u,\overline v] = [(\tilde v ^2+\tilde v h + f)/(-I_n x u),(x+V_n)\overline u -(h+\tilde v)].}
Using \eqst{h = -\tilde v +(V_{n-1}+x)u-w} and Lemma \ref{basicrel}, we calculate \eqst{ \tilde v ^2 +\tilde v h +f = -I_n(V_{n-1}+x)u^2 +I_n u w + (V_{n-1}+x)uv-I_{n-1}I_nV_{n-1}V_n u \vert{}^1\mathcal L _2\vert.}
Dividing by $-I_n u$ we obtain \begin{align*}
&(\tilde v ^2 + \tilde v h + f)/(-I_n u) \\
= \ & (V_{n-1}+x)u-w-(V_{n-1}+x)v/I_n+I_{n-1}V_{n-1}V_n \vert{}^1\mathcal L _2\vert \\
  = \ & (V_{n-1}+x) u_{V_n} -w_{V_n} \\ 
 = \ &xu_{V_n} + V_{n-1}u_{I_{n-1},V_n}.
 \end{align*}
 On the other hand we have \eqst{x \overline u = h_{I_n,V_n} \overset{(\ref{definitionh})}{=} (V_{n-1}+x)u_{V_n}-w_{V_n} = xu_{V_n}+V_{n-1}u_{I_{n-1},V_n}.}
 Thus showing the first equality. The second equality follows again from Lemma \ref{basicrel} and $h = \overline h$.
\end{proof}
\begin{corollary}
The inverse flow $\mathcal T^{-1}$ is induced by  \eqst{ -\mathfrak T =  \left\{ \begin{array}{c} ([x,-V_{\Pi} ],\lfloor \tfrac g 2 \rfloor) \text{, if $g$ even} \\  \  ([x,-V_{\Pi}],\lfloor \tfrac g 2 \rfloor +1) \text{, if $g$ odd}  \end{array}  \right. \in \mathrm{Jac}(\mathcal C).} 
\end{corollary}
\subsection{Discrete periodic Toda flow via the cyclic shift map}
First we explain how the time flow can be expressed in terms of the cyclic shift $\sigma$.
\begin{proposition}
\label{cyclictime}
The following relations hold
\begin{align*}
x \overline u &= x u +V_{n-1} u_{I_{n-1}} + v - V_n u^\sigma \\
			&= h +2v - I_n u  - V_n u^\sigma, \\
\overline v &= V_n \overline u + v - V_n u^\sigma.
\end{align*}
\end{proposition}
\begin{proof}
For the first assertion we calculate 
\begin{align*} x \overline u &= (\tilde v ^2 + \tilde v h +f)/(-I_n u) \\
&= (\tilde v ((V_{n-1}+x) u - w) + f) / (-I_n u) \\
&= ((-I_n u + v) (V_{n-1}+x) u - (-I_n u +v) w + f)/(-I_n u) \\
&= (V_{n-1}+x) u - w + ((V_{n-1}+x) u v - vw + f)/(-I_n u).
\end{align*}
On the other hand we have \begin{align*}
V_n u^\sigma &= (v^2 + vh +f)/(I_nu) \\
&=  ((I_n+V_{n-1}+x)uv-vw +f) / (I_n u) \\
&= v + ((V_{n-1}+x)uv-vw+f)/(I_n u).
\end{align*}
The result now follows with the observation $V_{n-1}u_{I_{n-1}} = V_{n-1}u-w$.
The second assertion follows from Lemma \ref{keyrelation} and Corollary \ref{easycor}.
\end{proof}
\noindent We further mention the following relations (many more combinations are possible as well)
\begin{proposition}
We have
\begin{align*}
x \overline u &= h -\tilde v - \tilde v ^\dagger, \\
u^\sigma &= \overline u -(\overline v - v)/V_n = \overline u _{I_1} + v/V_n \\
&= ((I_n+x) v + V_{n-1} v_{I_{n-1}})/(I_nV_n).
\end{align*}
\end{proposition}
\begin{proof}
The proofs are similar to the last proof.
\end{proof}

\section{How to recover the periodic box-ball flow}
We want to explain how to recover the periodic box-ball system from the discrete periodic Toda system. For this we quickly recall the definitions of the periodic box-ball and tropical periodic Toda systems and explain how they are naturally related with the discrete periodic Toda system.
\subsection{Periodic box-ball flow}
The periodic box-ball system is an integrable cellular automaton living on a discrete (finite) circle of length $N \geq 3$. It has been introduced in \cite{yura}, we refer to \cite{inoue} for a very nice survey. 
Recall that the periodic box-ball flow can be viewed as the map \eqst{\mathcal  B : \mathcal V_N \longrightarrow \mathcal V_N} explained in the Introduction, where the state space is given by $\mathcal V _N = \{ y \in \{0,1\}^N \vert \sum_{i=1}^N y_i < N/2 \}$. Remarkably, as shown in \cite{riemann}, the periodic box-ball flow can naturally be used to formulate the famous Riemann hypothesis.
\begin{remark}
One can view $\mathcal V _N$ as a crystal (in the sense of Kashiwara) and show that $\mathfrak B$ can be viewed as an element in an extend affine Weyl group $\mathcal W$ acting on $\mathcal V _N$, see \cite{inoue}.
\end{remark}

\subsection{Tropical periodic Toda flow}
\label{sec-trop-toda}
The tropical periodic Toda flow can be viewed as a tropical limit of the periodic discrete Toda flow (\ref{dpTflow}), as explained in \cite{tokihiro1,inoue}. 
It is given by the map \eqst{  \mathcal T^{\mathrm{trop}} : \mathbb R ^{2n} \dashrightarrow \mathbb R^{2n},}  \eqst{ (\bold{J}^t,\bold{W}^t) =  (J_1^t,\dots,J_n^t,W_1^t,\dots,W_n^t) \longmapsto (\bold{J}^{t+1},\bold{W}^{t+1}) ,} defined, for $i \in \{1,..,n\}$, $t \in \mathbb N$, by 
\eqsp{
J^{t+1}_i &= \mathrm{min}\{W^t_i,J_i^t + X_i^t\} ,\\ 
W^{t+1}_i &= J^t_{i+1}+W^t_i-J^{t+1}_i,}
where 
\eqst{X_i^t = \underset{{k \in \{0,1,..,n-1\}}}{\mathrm{max}}\{\sum_{l=1}^k J_{i-l}^t - W_{i-l}^t\},
}
with periodic boundary conditions 
\eqst{
J^t_{i+n} = J^t_i \text{ and } W^t_{i+n} = W^t_i.}

\begin{remark}
In the following $J_i$ and $W_i$ will always mean $J_i^t$ and $W_i^t$. 
\end{remark}

\begin{remark} 
We will study the restriction of $\mathcal T ^{\mathrm{trop}}$ to $\mathbb N^{2n} \subset \mathbb R^{2n}$, which is preserved by the flow and the relevant case for the periodic box-ball flow.
\end{remark}

\subsection{From periodic box-ball to tropical periodic Toda}
\label{sec-boxball-trop}
As explained in section 6.3 of \cite{inoue} there is a natural injective map \eqst{\eta : \mathcal V_N \longrightarrow \bigsqcup _{n=1}^{\lfloor N/2 \rfloor}\mathbb N^{2n}} from the phase space of the periodic box-ball system to the disjoint union of the phase spaces of the tropical periodic Toda system. Importantly, the subset \eqst{\mathcal V_{N,n} = \{y \in \mathcal V _N \ \vert \ \eta(y) \in \mathbb N^{2n}\}} is preserved by the flow $\mathcal B$. 
\begin{example}
We explain the definition of $\eta$ with an example.
Consider the state $y_{13} = 1100011000101 \in \mathcal V _{13}$ from the Introduction. First of all $n=4$ is given by the number of connected components of $1$'s. Then we record (from left to right) the sizes of the connected components of $1$'s and $0$'s, leading to $\bold J_{y_{13}} =(2,2,1,1)$ and $\bold Q_{y_{13}} = (3,3,1,0)$. 
\end{example}
\noindent Up to rotation every 
state in $y \in \mathcal V _{N}$ can be brought into this form, where the first and last box contain a ball. \\ \\
As explained in Example 6.19 \cite{inoue} the embedding is not compatible with the two different time flows, but the map becomes compatible after we divide by the cyclic shift map (with the usual conventions) \eqst{\sigma : (J_i,W_i) \longmapsto (J_{i+1},W_{i+1}).} More precisely, we have 
\begin{proposition}[Proposition 6.20 \cite{inoue}] \label{propinoue1} The following diagram is commutative 
\eqst{\xymatrix{ \mathcal V_{N,n} \ar[rr]^{\eta} \ar[d]^{\mathcal B} && \mathbb N^{2n}{/{\langle \sigma \rangle}} \ar[d]^{\mathcal T^{\mathrm{trop}}} \\ \mathcal V_{N,n} \ar[rr]^{\eta} && \mathbb N^{2n}{/{\langle \sigma \rangle}} }}
and the induced map \eqst{\eta: \mathcal V_{N,n} \longrightarrow \mathbb N^{2n}{ /{\langle \sigma \rangle}}} remains injective.
\end{proposition}
\begin{remark}
In particular we can calculate the period $\mathrm{per}(y)$ in terms of $\mathcal T^{\mathrm{trop}}$.
\end{remark}
\subsection{Lifting tropical Toda to discrete Toda}

In the literature, see \cite{tokihiro1,inoue}, to a state $(\bold{J},\bold{W}) \in \mathbb N^{2n} \subset \mathbb R ^{2n}$ of the tropical periodic Toda system one attaches the state $\mathfrak I _\epsilon = (\bold{I}_\epsilon,\bold{V}_\epsilon) \in \mathbb R^{2n}$ of the discrete periodic Toda system, depending on a (positive) parameter $\epsilon \in \mathbb R$, defined by \eqst{I_{\epsilon,i} = e^{-J_i / \epsilon} \text{ and } V_{\epsilon,i} =e^{-W_i / \epsilon}.}
In the tropical limit one recovers the original state $(\bold{J},\bold{W})$ via \eqst{\lim_{\epsilon \to 0}-\epsilon \, \mathrm{log}(I_{\epsilon,i}) = J_i\text{ and } \lim_{\epsilon \to 0}-\epsilon \, \mathrm{log}(V_{\epsilon,i})=W_i.}
This lift is compatible with the different time flows, i.e., we have the following commutative diagram (see Proposition 2.1 \cite{tokihiro1}) \eqst{\xymatrix{\mathbb N^{2n} \ar[rr]^{\mathfrak I _\epsilon} \ar[d]^{\mathcal T^{\mathrm{trop}}} && \mathbb R^{2n} \ar[d]^{\mathcal T} \\ \mathbb N^{2n} \ar[rr]^{\mathfrak I_\epsilon} && \mathbb R^{2n}.}}
Instead, we prefer to consider an algebraic lift. Introduce the variable $q$ and define \eq{\label{q-lift}\mathfrak I _q : (\bold{J},\bold{W})\in\mathbb N^{2n} \longmapsto (\bold{I}_q,\bold{V}_q)\in \mathbb Q(q)^{2n}} via \eqst{I_{q,i} = q^{J_i} \text{ and } V_{q,i} = q^{W_i}.} In this set-up the tropicalization procedure corresponds to the $q$-adic valuation $\nu_q$ on $\mathbb Q(q)$, i.e., we recover $(\bold{J},\bold{W})$ from $(\bold{I}_q,\bold{V}_q)$ simply by \eqst{J_i = \nu_q(I_{q,i}) \text{ and } W_i= \nu_q(V_{q,i}).}
In particular this means the discrete periodic Toda flow over $\mathbb Q(q)$ contains the tropical periodic Toda flow via 
\begin{proposition}[Reformulation of Proposition 2.1  \cite{tokihiro1}]
\label{reformulation}
The following diagram is commutative
\eqst{\xymatrix{\mathbb N^{2n} \ar@/^1.5pc/[rrrr]^{\mathrm{id}}  \ar[d]_{\mathcal T ^{\mathrm{trop}}} \ar[rr]_{\mathfrak I _q} && \mathbb Q (q)^{2n} \ar[d]^{\mathcal T} \ar[rr]_{\nu_q} && \mathbb N^{2n} \ar[d]^{\mathcal T ^{\mathrm{trop}}} \\ \mathbb N^{2n} \ar@/_1.5pc/[rrrr]_{\mathrm{id}} \ar[rr]^{\mathfrak I _q} && \mathbb Q(q)^{2n}\ar[rr]^{\nu_q} && \mathbb N^{2n}.}}
\end{proposition}

\subsection{All together}
We define $\mathcal V_{N,\mathcal C} = \eta^{-1}(\nu_q(\mathcal R _\mathcal C))$.
Summarising everything said so far, we obtain the following description of the periodic box-ball flow in terms Mumford's $\jac$ and the Gauß composition law. 
\begin{theorem}
\label{maintheorem2}
The following diagram is commutative
\eqst{\xymatrix{\mathcal V_{N,\mathcal C}\ar[rr]^{ \eta} \ar[d]^{\mathcal B}&& \mathbb N^{2n}{/{\langle \sigma \rangle}} \ar[rr]^{\mathfrak I _q} \ar[d]^{\mathcal T^{\mathrm{trop}}} && \mathbb Q(q)^{2n}{/{\langle \sigma \rangle}} \ar[rr]^{\Psi_\mathcal C} \ar[d]^{\mathcal T} && \mathrm{Jac}(\mathcal C)/\langle\mathfrak{pa}\rangle \ar[d]^{\cdot \boxplus \mathfrak T} \\  \mathcal V_{N,\mathcal C}\ar[rr]^{ \eta}&& \mathbb N^{2n}{/{\langle \sigma \rangle}} \ar[rr]^{\mathfrak I _q} && \mathbb Q(q)^{2n}{/{\langle \sigma \rangle}} \ar[rr]^{\Psi_\mathcal C}&&\mathrm{Jac}(\mathcal C)/\langle\mathfrak{pa}\rangle }}
and composition of the horizontal maps is injective.
\end{theorem}

\begin{remark}
On the quotient $\jac / \langle \mathfrak{pa} \rangle$ all subtleties of the Gauß composition law for real hyperelliptic curves concerning the divisors at infinity disappear.
\end{remark}

\section{Integrality result and applications}
\noindent Obviously, one can now wonder whether the discrete periodic Toda flow has some integrality properties? Somewhat surprisingly we show that the answer is positive and leads, for example, to a new $p$-adic description of the periodic box-ball flow. This point of view will be studied thoroughly in follow-up work \cite{yalkinoglu}.
\begin{theorem}
\label{integrality}
The discrete periodic Toda flow induces a flow \eqst{\mathcal T_{\mathrm{integral}} : \mathbb Z[q]_{(q)}^{2n} \dashrightarrow \mathbb Z[q]_{(q)}^{2n},} which recovers the periodic box-ball flow using the $q$-adic valuation $\nu_q$. Here $\mathbb Z[q]_{(q)}$ denotes the localization of $\mathbb Z[q]$ at the ideal $(q)$, which is a local ring\footnote{In particular we have $q^{-1} \notin \mathbb Z[q]_{(q)}$.}.
\end{theorem}
\begin{proof}
For $(\bold{I}^t,\bold{V}^t) \in \mathbb Z[q]_{(q)}^{2n} $ we need to show that \eqst{\nu_q(I^{t+1}_i),\nu_q(V^{t+1}_i) \geq 0} for all $i \in \{1,\dots,n\}$, if the state $(\bold{I}^{t+1},\bold{V}^{t+1})$ is well-defined. This follows from the explicit formulas for the inverse Abel-Jacobi map in Proposition \ref{inversepsi}, as we have \eqst{\nu_q(I_{\Pi}),\nu_q(V_{\Pi}) \geq \nu_q(\sigma^k(u^{t+1}_x))} for all $k \in \{1,\dots,n\}$.
\end{proof}
\begin{corollary}
\label{padic}
The substitution $q \mapsto p$, for $p$ a (big enough) prime number, leads to the flow
\eqst{\mathcal T _p : \mathbb Z_{(p)}^{2n} \dashrightarrow \mathbb Z_{(p)}^{2n},} which recovers the periodic box-ball flow using the $p$-adic valuation $\nu_p$.
By continuity, we further obtain the flow \eqst{\mathcal T _p : \mathbb Z _p^{2n} \dashrightarrow \mathbb Z _p^{2n}.}
\end{corollary}
\begin{remark}
In particular we see that the discrete periodic Toda flow can be defined on (certain) local rings which are not fields! This integrality property seems to have been overlooked so far.
\end{remark}
\section{Outlook}
\subsection*{p-adic discrete periodic Toda flow}
Our integrality result allows to study the dynamics of the discrete periodic Toda system (and in particular the dynamics of the periodic box-ball system) using $p$-adic methods! In particular, in follow-up work we will study division polynomials for real models of hyperelliptic curves (Toda curves) over $\mathbb Z_p$. In this $p$-adic set-up many new tools can be used to tackle this problem, such as $p$-adic differential equations (see \cite{eid}) or the theory of ($p$-adic) formal group laws.
\subsection*{Berkovich spaces over $\mathbb Z$}
We finish our paper with the following somewhat speculative
\begin{question}
Is it possible to define the discrete periodic Toda flow on some (natural) Berkovich space defined over $\mathbb Z$?
\end{question}

\appendix
\section{Monodromy matrix approach}
In \cite{tokihiro1} a different approach to the computation of $h_n(x)$ and $f_n$ was given in terms of the so-called monodromy matrix \eqst{ M_n = \begin{bmatrix} m^1_{n+1} & m^1_n \\ m^2_{n+1} & m^2_n\end{bmatrix},} defined by the recursion (with index $j$ subject to the usual condition (\ref{periodicbc})) \eqst{m_{j} = (I_{j+1}+V_j +x) m_{j-1}- I_j V_j m_{j-2} } together with the two initial conditions \eqst{(m^1_1,m^1_0) = (1,0) \text{ and } (m^2_1,m^2_0) = (0,1),}
leading to the two relations \eqst{h_n(x) = \mathrm{tr}(M_n) =  m^1_{n+1}+m^2_{n} } and \eqst{f_n = \vert M_n \vert = m^1_{n+1}m^2_{n}-m^1_n m^2_{n+1}.}
The connection to our set-up is given by the next
\begin{lemma}
For $n \geq 2$, we have \begin{align*}(m^1_n,m^2_n) &= (u_n^{\sigma^{2-n}},-v_n^{\sigma^{2-n}}), \\
(m^1_{n+1},m^2_{n+1}) &= ((h_n+v_n)^{\sigma^{2-n}},(-I_nV_n u_n^\sigma)^{\sigma^{2-n}}).\end{align*}
\end{lemma}
\begin{proof}
We leave this as an exercise. 
\end{proof}

\section{Cantor's division polynomials}

\noindent We want to explain the following theorem due to Cantor
\begin{theorem}[\cite{cantordivision} Theorem 8.7]
\label{cantortheorem}
Let $\mathcal D: Y^2 = F(X)$ be an imaginary hyperelliptic curve with 
$F$ monic, $\mathrm{deg}(F) = 2g+1$ and $(a,b) \in \mathcal D$. Then, for $r \geq g+1$, there are explicit polynomials $\Delta_r(z),E_r(z)$ such that\eqst{r\cdot [X-x,y] = [\Delta_r(\tfrac{a-X}{4b^2},E_r(\tfrac{a-X}{4b^2})] \in \mathrm{Jac(\mathcal D)}.}
\end{theorem}
\noindent Roughly speaking the idea is to use the Padé approximants of $\sqrt{F(x)}$ for the construction of $\Delta_r(z)$ and $E_r(z)$.
\subsection{Discrete open Toda flow}
Set $E(z) = F(a-z)$ and consider the formal power series \eqst{S(z) = \sum_{i\geq 0} s_i z^i = (-1)^{g+1} \sqrt{E(z)}.} Define the Hankel determinants \eqst{h_{m,n} = (-1)^{\binom{n}{2}}\vert (s_{m+i-j})_{i,j=1}^n \vert.}
Then the $h_{m,n}$ satisfy the following fundamental relation (see \cite{gragg} Theorem 5.1), first discovered by Frobenius, namely \eq{\label{openToda}h_{m-1,n}h_{m+1,n}=h_{m,n-1}h_{m,n+1}+h_{m,n}^2.}
\begin{remark}
The relation (\ref{openToda}) is equivalent to the open discrete Toda equation in the $\tau$-function formalism, as explained in \cite{gragg,faybusovich}.
\end{remark}
\subsection{Cantor's division polynomials}
Following \cite{cantordivision}, let $r \in \mathbb N$, set $m_r = \lfloor \tfrac{r+g}{2} \rfloor ,n_r = \lfloor \tfrac{r-g-1}{2}\rfloor$ and define 
\eqst{f_r = \left\{ \begin{array}{cc}  h_{m_r,n_r} & ,r \geq g \\  \  0 & , r<g  \end{array} .\right.}
The polynomials $A^i_r(z)$, for $i \in \{1,2\}$, are defined by the recursion \eqst{f_{r-1}A^i_{r+1}(z) = f_r A^i_r(z)-zf_{r+1}A^i_{r-1}(z)} with initial conditions \eqst{(A^1_{g-1},A^1_g) = (-z^{g-1},-z^g) \text{ and } (A^2_{g-1},A^2_g) = (0,0).}
The polynomials $A^i_r(z)$ are Padé approximants of $\sqrt{E(z)}$.
Further, for $r\geq g+1$, we define \eqst{D_r(z) = -(A^1_r(z)^2 - A^2_r(z)^2 E(z))/z^r} and \eqst{\tilde E (z) = \tfrac{z f_{r-1}f_{r+1}}{2 f_r^2}(\tfrac{D_{r+1}(z)}{f_{r+1}^2}-\tfrac{D_{r-1}(z)}{f_{r-1}^2}) \text{ mod } D_r(z).}
Finally, we can define the polynomials appearing in Cantor's theorem by \eqst{\Delta_r(z) = (2b)^{2\nu_r} D_r(4b^2 z) \text{ and } E_r(z) = \tilde E_r(4 b^2 z),} where $\nu_r = \binom{r}{2}-\binom{g}{2}.$

\subsection{Back to our set-up}
In light of Theorem \ref{lineartime}, the discrete periodic Toda flow is essentially reduced to understand (for $t \in \mathbb N$) \eqst{t \cdot \mathfrak T \in \jac.} Of course, this can be calculated via the Gauß composition law, but Cantor's division polynomials are even more efficient (cf., \cite{cantordivision}). On the imaginary model $\mathcal C _\alpha$ the element $\mathfrak T$ becomes \eqst{\mathfrak T _\alpha = [x_\alpha-\tfrac 1 \alpha, \tfrac{V_{\Pi}-I_{\Pi}}{\beta (-\alpha)^{g+1}}] \in \mathrm{Jac}(\mathcal C _\alpha)} to which we can apply Cantor's division polynomials (for $t > g$). 

\begin{remark}
As far as the author knows, there is surprisingly no direct (algebraic) approach to division polynomials for real hyperelliptic curves in the literature. Interestingly, $p$-adic methods might pave the way to such a theory.
\end{remark}

\section{Tropical Mumford Jacobian}
\label{tropicalmumford}
\noindent We continue to work with the lift $\mathfrak I _q : \mathbb N^{2n} \longrightarrow \mathbb Q(q)^{2n}$ defined in (\ref{q-lift}). For background on tropical polynomials we refer to \cite{nobe}.
\begin{definition}
Given $F(x) = \sum_{i=0}^n f_i x^i \in \mathbb Q(q)[x]$, we define its tropical polynomial by \eqst{F^{\mathrm{trop}}(x) = \mathrm{min}\{\nu_q(a_0),\nu_q(a_1)+x,\dots,\nu_q(a_n) + nx\}.}
\end{definition}
\noindent From Lemma \ref{subfree} we know that the polynomials describing the Mumford representations $[u,v]$ of $\jac$ are all subtraction-free. This leads to
\begin{definition}
Define the tropical Mumford Jacobian to be \eqst{\jac^{\mathrm{trop}} = \{([u^{\mathrm{trop}},v^{\mathrm{trop}}],n) \ \vert \ ([u,v],n)\in \jac\}} and the tropical isolevel set via \eqst{\mathcal R_\mathcal C^{\mathrm{trop}} = \{ \nu_q((\bold{I},\bold{V})) \vert (\bold{I},\bold{V}) \in \mathcal R _{\mathcal C} \} \subset \mathbb N^{2n}.} 
\end{definition}

\begin{proposition}[Tropical inverse Abel-Jacobi map]
There is a natural tropicalization \eqst{(\Psi^{-1}_\mathcal C)^{\mathrm{trop}} : \mathrm{Im}(\Psi_\mathcal C)^{\mathrm{trop}} \longrightarrow \mathcal R_\mathcal C ^{\mathrm{trop}}} of the algebraic inverse Abel-Jacobi map from Proposition \ref{inversepsi} such that the following diagram 
\eqst{\xymatrix{ \mathcal R _\mathcal C \ar[rr]^{\Psi_\mathcal C^{-1}}  \ar[d]^{{}^\mathrm{trop}} && \mathrm{Im}(\Psi_\mathcal C) \ar[d]^{{}^\mathrm{trop}} \\  \mathcal R _\mathcal C ^{\mathrm{trop}} \ar[rr]^{(\Psi^{-1}_{\mathcal C})^{\mathrm{trop}} }&&  \mathrm{Im}(\Psi_\mathcal C) ^{\mathrm{trop}} }}
commutes.
\end{proposition}
\begin{proof}
Using Lemma \ref{subfree} it is immediate that the explicit formulas from Proposition \ref{inversepsi} can be tropicalized such that the diagram commutes.
\end{proof}
\begin{proposition}[Linearization of tropical Toda flow]
We can define a map \eqst{\boxplus \mathfrak T ^{\mathrm{trop}} : \jac^{\mathrm{trop}} \longrightarrow \jac ^{\mathrm{trop}}} such that the following diagram is commutative
\eqst{\xymatrix{ \jac \ar[d]^{\cdot \boxplus \mathfrak T} \ar[rr]^{{}^{\mathrm{trop}}} && \jac^{\mathrm{trop}} \ar[d]^{\boxplus \mathfrak T ^{\mathrm{trop}}}\\  \jac \ar[rr]^{{}^{\mathrm{trop}}} &&\jac^{\mathrm{trop}}.} }
\end{proposition}
\begin{proof}
For this we exploit again Lemma \ref{subfree} and use transport of structure.
\end{proof}
\begin{remark}
It should be interesting to investigate the Gauß composition law and the theory of division polynomials in the framework of tropical curves and Jacobians. We do not pursue this direction in this paper.
\end{remark}

\bibliographystyle{plain}
\bibliography{lineardpT_v2}

\end{document}